\documentclass[12pt]{amsart}
\usepackage[alphabetic]{amsrefs}
\usepackage{mathdots, blkarray}
\usepackage{multicol}
\usepackage{graphicx}
\usepackage{amscd}
\usepackage{upgreek}
\usepackage{stmaryrd}
\usepackage{longtable}
\usepackage[T1]{fontenc}
\usepackage{latexsym, amsmath, amssymb, amsthm}
\usepackage[svgnames]{xcolor}
\usepackage{rsfso}
\usepackage[mathscr]{eucal}
\usepackage{mathtools}
\usepackage{mathptmx}
\usepackage{titletoc}
\usepackage{wrapfig}
\usepackage{float}
\usepackage{xypic}
\usepackage{microtype}
\usepackage{dsfont}
\usepackage{xcolor}
\usepackage{color}
\usepackage[colorlinks = true,
            linkcolor  = DarkBlue,
            urlcolor   = DarkRed,
            citecolor  = DarkGreen]{hyperref}
\usepackage{hyperref}
\allowdisplaybreaks	
\usepackage[centering, includeheadfoot, hmargin=1.0in, tmargin=0.5in, 
  bmargin=0.6in, headheight=6pt]{geometry}
  
\usepackage{listings}
\usepackage{xcolor}

\newtheorem{theorem}{Theorem}[section]
\newtheorem{prop}[theorem]{Proposition}

\newtheorem{lem}[theorem]{Lemma}
\newtheorem{coro}[theorem]{Corollary}
\newtheorem{conj}[theorem]{Conjecture}
\newtheorem{thm}[theorem]{Theorem}

\newtheorem{rem}[theorem]{\rm\textsc{Remark}}
\newtheorem{exam}[theorem]{\rm\textsc{Example}}

\newcommand{\ideal}[1]{\ensuremath{\left\langle #1 \right\rangle}}

\newcommand{\bslash}{\kern-0.1em\texttt{\scalebox{0.6}[1]{/}}\kern-0.15em \texttt{\scalebox{0.6}[1]{/}}}

\DeclareMathOperator{\GL}{GL}

\DeclareMathOperator{\Tr}{Tr}
\DeclareMathOperator{\cha}{char}

\newcommand{\A}{\mathcal{A}} 
\newcommand{\B}{\mathcal{B}} 
\newcommand{\D}{\mathcal{D}} 
\newcommand{\C}{\mathbb{C}}

\newcommand{\N}{\mathbb{N}} 
\newcommand{\R}{\mathbb{R}} 
\newcommand{\F}{\mathbb{F}} 
\newcommand{\p}{\mathfrak{p}} 
 
\newcommand{\K}{\mathbb{K}} 
\newcommand{\LA}{\Longleftarrow} 
\newcommand{\RA}{\Longrightarrow} 
 
\newcommand{\ra}{\longrightarrow}

\newcommand{\hbo}{$\hfill\Diamond$} 

\begin{document}
\title{Geometry of Yang-Baxter matrix equations over finite fields} 
\def\shorttitle{Geometry of Yang-Baxter matrix equations over finite fields}

\author{Yin Chen}
\address{School of Mathematics and Physics, Key Laboratory of ECOFM of 
Jiangxi Education Institute, Jinggangshan University,
Ji'an 343009, Jiangxi, China \& Department of Finance and Management Science, University of Saskatchewan, Saskatoon, SK, Canada, S7N 5A7}
\email{yin.chen@usask.ca}

\author{Shaoping Zhu}
\address{School of Mathematics and Physics, Key Laboratory of ECOFM of 
Jiangxi Education Institute, Jinggangshan University, Ji'an 343009, Jiangxi, China}
\email{zhushaoping@jgsu.edu.cn}

\begin{abstract}
Let $A$ be a $2\times 2$ matrix over a finite field and consider the Yang-Baxter matrix equation $XAX=AXA$ with respect to $A$. We use a method of computational ideal theory to explore the geometric structure of the affine variety of all solutions to this equation. In particular, we exhibit all solutions explicitly and determine  cardinality formulas for these varieties. 
\end{abstract}

\date{\today}
\maketitle \baselineskip=16.2pt

\dottedcontents{section}[1.16cm]{}{1.8em}{5pt}
\dottedcontents{subsection}[2.00cm]{}{2.7em}{5pt}

\section{Introduction}
\setcounter{equation}{0}
\renewcommand{\theequation}
{1.\arabic{equation}}
\setcounter{theorem}{0}
\renewcommand{\thetheorem}
{1.\arabic{theorem}}

\noindent Let $\K$ be a field of any characteristic and $n\in\N^+$. Suppose that $M_n(\K)$ denotes the vector space of all $n\times n$ matrices over $\K$.  Given a matrix $A\in M_n(\K)$, the following equation 
\begin{equation}\tag{YBME}
\label{YBME}
XAX=AXA
\end{equation}
is called the \textit{Yang-Baxter matrix equation} with respect to $A$ over $\K$.  This matrix equation occupies a prominent place in both pure mathematics and mathematical physics, which originally stems from 
Yang’s work on the one-dimensional many-body problem in \cite{Yan67} and Baxter’s research in two-dimensional classical statistical mechanics in \cite{Bax72}, and its study has substantial ramifications in quantum groups, nonassociative algebras, and differential operator theory; see for example, \cite{CJdR10, Hiv00}, and \cite{BFP99}.

We write $\D_A(\K)$ for the affine variety of all  $n\times n$ matrix  solutions  to (\ref{YBME}). Apparently, if $A=0$, then $\D_A(\K)=M_n(\K)$. If $A\neq 0$, then  the zero matrix and the matrix $A$ both are solutions to (\ref{YBME}) and thus $|\D_A(\K)|\geqslant 2$. We say that a solution $X\in\D_A(\K)$ is \textit{nontrivial} if $X\neq 0$ and $X\neq A$. A central task in the study of \ref{YBME} is to characterize
the shape of a generic solution in $\D_A(\K)$ and to construct new nontrivial solutions within $\D_A(\K)$ for specific applications. 
There are extensive research conducted on this task in the characteristic zero case;  for recent developments over $\C$ or $\R$, see \cite{CCY19, Lu22, WWW24, GZ24}, and \cite{HM25}. In this article, we focus on the finite field case (i.e., $\K=\F_q$) and our primary objective is to explore the geometric structure of the affine variety $\D_A(\F_q)$ via capitalizing on 
a method from computational ideal theory.

Characterizing solutions of a matrix equation over finite fields and determining explicit formulas for their cardinalities 
have been a classical topic with a long history, dating back to the works of Carlitz and Hodges; see for example, \cite{Car73, Hod57, Hod58, Hod64}, and \cite{Hod66}. In contrast to the characteristic zero case, where $\D_A(\K)$ admits infinitely many nontrivial solutions, the finite field setting yields only finitely many solutions. This finiteness provides two notable advantages: it facilitates the explicit computation of the cardinality of $\D_A(\F_q)$, and it enables applications to various areas such as combinatorics, algebra, and the theory of combinatorial designs in statistics; see \cite{WZ01a} and \cite{WZ01b}.
In earlier work \cite{CZ23}, we used algebraic invariant theory to study the geometry of $\D_A(\F_q)$ when $A$ is a scalar matrix. This work has inspired several recent studies, including \cite{PDK22,CZ26, CR25} and \cite{Coh25}. The present article aims to extend this investigation to the case where $A$ ranges over all $2\times 2$-matrices. 

Throughout we assume that $\F_q$ denotes a finite field of order $q=p^s$ for some prime $p$ and $s\in\N^+$.
Given two matrices $A,B\in M_2(\F_q)$, we observe that if $A$ and $B$ are similar over $\F_q$, then $|\D_A(\F_q)|=|\D_B(\F_q)|$ and 
solutions in $\D_B(\F_q)$ can be obtained from elements in $\D_A(\F_q)$ via a conjugacy action; see Proposition \ref{prop2.1}
and Corollary \ref{coro2.2}.
This observation allows us to restrict our attention to  the affine variety $\D_A(\F_q)$, where $A$ denotes the rational canonical form of $B$. We will see below that the rational canonical form $A$ must be one of the following four types:
\begin{equation}
\label{ }
\left\{A_1:=\begin{pmatrix}
   c_1   &  0  \\
    0  &  c_2
\end{pmatrix}, ~~A_2:=\begin{pmatrix}
   c   &  1  \\
    0  &  c
\end{pmatrix},~~ A_3:=\begin{pmatrix}
   0   &  -b  \\
    1  &  a
\end{pmatrix}, ~~ A_4:=\begin{pmatrix}
   c   &  0  \\
    0  &  c
\end{pmatrix}\right\}
\end{equation}
where $a,b,c,c_1\neq c_2$ are elements in $\F_q$ and $x^2-ax+b$ is irreducible in $\F_q[x]$; see Proposition \ref{prop2.5}.  Note that the cardinality and geometry of $\D_{A_4}(\F_q)$ have been well understood, as shown in \cite{CZ23}. Thus, in this article, we mainly focus on the geometric structures and cardinalities for  $\D_{A_1}(\F_q), \D_{A_2}(\F_q)$, and $\D_{A_3}(\F_q)$.

The first main result can be summarized as follows.

\begin{thm}\label{mt1}
{\rm (1)} If exactly one of $\{c_1,c_2\}$ is zero, then $\D_{A_1}(\F_q)$ is a 2-dimensional affine variety consisting of two 2-dimensional irreducible components and one 1-dimensional irreducible component. In particular, $|\D_{A_1}(\F_q)|=2q^2$.

{\rm (2)} Assume that $c_1\neq c_2\in\F_q^\times$ and define $\updelta:=c_1^2  -c_1c_2+ c_2^2$. If $\updelta=0$, then 
$|\D_{A_1}(\F_q)|=2q+2$ and $\D_{A_1}(\F_q)$ is a 1-dimensional affine variety consisting of three isolated points and two 1-dimensional irreducible components. If $\updelta\neq0$, then 
$|\D_{A_1}(\F_q)|=q+3$ and $\D_{A_1}(\F_q)$ is a 1-dimensional affine variety consisting of four isolated points and one 1-dimensional irreducible component.
\end{thm}
\noindent The proof of Theorem \ref{mt1} can be given by together Propositions \ref{prop3.1}, \ref{prop3.2} and \ref{prop3.3}.

To understand the geometric structure of $\D_{A_2}(\F_q)$, our argument will be separated into two cases: $c=0$ and $c\neq 0$. The following theorem is our second major result, which could be proved by combining the proofs of Propositions \ref{prop4.1} and \ref{prop4.2}.

\begin{thm}\label{mt2}
{\rm (1)}  If $c=0$, then $\D_{A_2}(\F_q)$ is of cardinality $2q^2-q$ and is a 2-dimensional affine variety consisting of 
two 2-dimensional irreducible components. 

{\rm (2)} If $c\neq 0$, then $\D_{A_2}(\F_q)$  is a 1-dimensional affine variety consisting of two isolated points and 
one 1-dimensional irreducible component. In particular, $|\D_{A_2}(\F_q)|=q+2$.
\end{thm}

To state the third main result, we write 
\begin{equation}
\label{ }
\Delta:=a^2-4b
\end{equation}
for the discriminant of $x^2-ax+b$.  The following theorem demonstrates that for all odd characteristic cases, there are no nontrivial solutions in $\D_{A_3}(\F_q)$  provided that $\Delta= -b$.  

\begin{thm}\label{mt3}
Assume that $\cha(\F_q)\neq 2$. If $\Delta= -b$, then $|\D_{A_3}(\F_q)|=2$.
\end{thm}
\noindent We also propose a conjecture about the cardinality of $\D_{A_3}(\F_q)$ for the case $\Delta\neq -b$.

We organize this article as follows.  In Section \ref{sec2}, we provide some fundamental properties on matrix similarity and conjugacy actions that are helpful in understanding the geometry of $\D_{A}(\F_q)$ and its quotient spaces. In particular, we give a detailed discussion about rational canonical forms of $2\times 2$ matrices over $\F_q$. Sections \ref{sec3}, \ref{sec4}, and \ref{sec5} are devoted to proofs of Theorem \ref{mt1}, \ref{mt2}, and \ref{mt3}, respectively. In Section \ref{sec6}, we present some examples and explicit calculations for small primes $q$. Moreover, together with Remark \ref{rem6.4}, a conjecture on $|\D_{A_3}(\F_q)|$ (when $\Delta\neq -b$) is raised in this section.

\vspace{2mm}
\noindent \textbf{Summary of contributions}. 
The following is a summary of our contributions in the article:
\begin{enumerate}
  \item We give a complete understanding for the variety $\D_A(\F_q)$ of solutions to the Yang-Baxter $2\times 2$-matrix equation $XAX=AXA$ over a finite field $\F_q$, unless a special case of $\Delta\neq -b$ as stated in Conjecture \ref{conj6.5} below. 
  \item We derive the explicit cardinality formula of the solution variety $\D_A(\F_q)$ and characterize the geometric structure of  $\D_A(\F_q)$ as an affine variety.  All main results are stated in Theorem \ref{mt1}, \ref{mt2}, and \ref{mt3}.
  \item Enhancing existing methods, we use a method from computational ideal theory to describe the geometric structure of
  $\D_A(\F_q)$. 
  \item Several explicit examples for small prime fields are presented in Section \ref{sec6}. 
\end{enumerate} 

\vspace{2mm}
\noindent\textbf{Conventions.}  Throughout this article,  $\K$ denotes a field of any characteristic and $\F_q$ denotes a finite field of order $q$. We always assume that $A$ is a nonzero matrix over either $\K$ or $\F_q$. We write $R:=\F_q[x_{11},x_{12},x_{21},x_{22}]$ for the polynomial ring over $\F_q$ and  denote by $\ideal{S}$  the ideal of $R$ generated by a subset $S\subseteq R$.

\section{General Properties} \label{sec2}
\setcounter{equation}{0}
\renewcommand{\theequation}
{2.\arabic{equation}}
\setcounter{theorem}{0}
\renewcommand{\thetheorem}
{2.\arabic{theorem}}

\noindent This preliminary section presents some basic properties on matrix similarity and conjugacy actions. These properties allow us to reduce the study of $\D_A(\K)$ to the case where $A$ denotes a rational canonical form over $\F_q$. A detailed discussion of rational canonical forms of size $2$ over $\F_q$  is also provided. 

We denote by $\GL_n(\K)$ the general linear group of degree $n$ over $\K$.

\begin{prop}\label{prop2.1}
Let $A$ and $B$ be two similar matrices of degree $n$ over $\K$. Then there exists a one-to-one correspondence between
$\D_A(\K)$ and $\D_B(\K)$.  
\end{prop}

\begin{proof}
Suppose that $B=P^{-1}AP$ for some $P\in\GL_n(\K)$. The map 
\begin{equation}\label{eq2.1}
\upvarphi_P: \D_A(\K)\ra\D_B(\K)
\end{equation}
defined by $X\mapsto P^{-1}XP$ for all $X\in \D_A(\K)$, is a bijective map.  The first fact we have to show is that
$\upvarphi_P(X)=P^{-1}XP\in\D_B(\K)$ for all $X\in \D_A(\K)$. In fact,  
\begin{eqnarray*}
(P^{-1}XP)B(P^{-1}XP)&=&P^{-1}X(PBP^{-1})XP=P^{-1}(XAX)P\\
&=&P^{-1}(AXA)P=(P^{-1}AP)(P^{-1}XP)(P^{-1}AP)\\
&=&B(P^{-1}XP)B.
\end{eqnarray*}
Hence, $P^{-1}XP\in\D_B(\K)$ for all $X\in \D_A(\K)$. Note that the map $\upvarphi_P$ is given by the conjugacy action, thus it is bijective. This completes the proof.
\end{proof}

\begin{coro} \label{coro2.2}
If $A$ and $B$ are similar over $\K$, then $|\D_A(\K)|=|\D_B(\K)|$.  
\end{coro}

Fix a matrix $A\in M_n(\K)$ and consider the conjugacy action of $\GL_n(\K)$ on $M_n(\K)$:
\begin{equation}\label{ }
(Q,Y)\mapsto QYQ^{-1}
\end{equation}
for all $Q\in \GL_n(\K)$ and $Y\in M_n(\K)$. In general,  $\D_A(\K)$ is not stable under this action of $\GL_n(\K)$. 
However, if we write $G_A$ for the stabilizer subgroup of $A$ in $\GL_n(\K)$, i.e.,
$$G_A:=\{Q\in \GL_n(\K)\mid QAQ^{-1}=A\},$$
then the following result shows that $\D_A(\K)$ is stable under the conjugacy action of $G_A$, extending \cite[Proposition 2.1]{CZ23}.

\begin{prop} \label{prop2.3}
For any $A\in M_n(\K)$, $\D_A(\K)$ is stable under the conjugacy action of $G_A$. 
\end{prop}

\begin{proof}
It suffices to show that $QXQ^{-1}\in \D_A(\K)$ for all $Q\in G_A$ and $X\in \D_A(\K)$. Since $QA=AQ$, it follows that
\begin{eqnarray*}
(QXQ^{-1})A(QXQ^{-1})& = & QXQ^{-1}(QA)XQ^{-1}\\
 & = & Q(XAX)Q^{-1}=Q(AXA)Q^{-1}\\
 &=&A(QXQ^{-1})A.
\end{eqnarray*}
Here, the last equation holds because $A$ and $Q^{-1}$ commute as well. 
\end{proof}

\begin{rem}{\rm
The geometry of a quotient space $\D_A(\K)\bslash G_A$ is not the main topic of this article. However, as we have seen in \cite{CZ23}, stabilizer subgroups and the stability of $\D_A(\K)$ would play a key role in understanding the geometric structure of $\D_A(\K)\bslash G_A$. 
\hbo}\end{rem}

The rest of this section is to explore all rational canonical forms of $2\times 2$ matrices over $\F_q$. 
Let us recap two fundamental facts in Linear Algebra that (1) two matrices are similar if and only if they have the same rational canonical form; (2) The rational canonical form of a matrix is determined by the minimal polynomial of this matrix. Thus,  to find all possible rational canonical forms in $M_2(\F_q)$, we need to consider characteristic polynomials and minimal polynomials of matrices. 

Suppose that $B\in M_2(\F_q)$ denotes an arbitrary matrix. We write $m_B(x)$ and $c_B(x)$ for the minimal polynomial and the characteristic polynomial of $B$ respectively. Note that  $m_B(x)$ and $c_B(x)$ both are monic polynomials of degree 2 in $\F_q[x]$. In particular, $m_B(x)$ divides $c_B(x)$. 

We may assume that $c_B(x)=x^2-ax+b$, where $a=\Tr(B)$ and $b=\det(B)$. We denote by $r_B$ the number of distinct  roots of $c_B(x)$ in $\F_q$. Clearly, $0\leqslant r_B\leqslant 2$.  Our arguments can be separated into the following three cases:

\textsc{Case 1}: $r_B=2$. We may assume that $c_B(x)$ has two distinct roots $c_1\neq c_2$ in $\F_q$. Then $c_B(x)=(x-c_1)(x-c_2)$ splits in $\F_q$ and  $m_B(x)=c_B(x)=(x-c_1)(x-c_2)$. Thus the rational canonical form of $B$
is $$A_1=\begin{pmatrix}
   c_1   &  0  \\
    0  &  c_2
\end{pmatrix}$$
which means that $B$ is diagonalizable.

\textsc{Case 2}: $r_B=1$. In this case, $c_B(x)=(x-c)^2$ must split in $\F_q$ and it has one repeated root $c$ in $\F_q$. There are two subcases: $m_B(x)=x-c$ and $m_B(x)=(x-c)^2$. For the first subcase, the rational canonical form of $B$
is $$A_4=\begin{pmatrix}
   c   &  0  \\
    0  &  c
\end{pmatrix}.$$
In the second subcase, the rational canonical form of $B$ will be a Jordan block of size 2:
$$A_2=\begin{pmatrix}
   c   &  1  \\
    0  &  c
\end{pmatrix}.$$

\textsc{Case 3}: $r_B=0$. In this case, $c_B(x)$ is irreducible, and so $m_B(x)=c_B(x)=x^2-ax+b$. Thus the rational canonical form of $B$
is $$A_3=\begin{pmatrix}
   0   &  -b  \\
    1  &  a
\end{pmatrix}.$$

This also proves the following result. 

\begin{prop} \label{prop2.5}
The rational canonical form of a matrix $B\in M_2(\F_q)$ must be  one of the following four types:
\begin{equation}
\label{ }
\left\{\begin{pmatrix}
   c_1   &  0  \\
    0  &  c_2
\end{pmatrix},  ~~\begin{pmatrix}
   c   &  1  \\
    0  &  c
\end{pmatrix},~~ \begin{pmatrix}
   0   &  -b  \\
    1  &  a
\end{pmatrix},~~\begin{pmatrix}
   c   &  0  \\
    0  &  c
\end{pmatrix}\right\}
\end{equation}
where $a,b,c,c_1\neq c_2$ are elements in $\F_q$ and $x^2-ax+b$ is irreducible in $\F_q[x]$.
\end{prop}

\section{Geometry of $\D_{A_1}(\F_q)$} \label{sec3}
\setcounter{equation}{0}
\renewcommand{\theequation}
{3.\arabic{equation}}
\setcounter{theorem}{0}
\renewcommand{\thetheorem}
{3.\arabic{theorem}}

\noindent In this section, we deal with the situation of $A=A_1=\begin{pmatrix}
   c_1   &  0  \\
    0  &  c_2
\end{pmatrix}$
for two distinct elements $c_1$ and $c_2$ in $\F_q$. Our main goal  is to give a proof to Theorem \ref{mt1}. Precisely, we would like to study the irreducible components of $\D_{A_1}(\F_q)$ and determine the cardinality of $\D_{A_1}(\F_q)$. Note that
$c_1\neq c_2$, we may divide the proof  of Theorem \ref{mt1} into two cases: (1)  exactly one of $\{c_1,c_2\}$ is zero; (2) $c_1$ and $c_2$ both are nonzero.

Let us begin with an arbitrary solution $X=\begin{pmatrix}
   x_{11}   &  x_{12}  \\
    x_{21}  &  x_{22}
\end{pmatrix}\in\D_{A_1}(\F_q)$. Substituting $X$ to (\ref{YBME}) obtains 
$$XA_1X-A_1XA_1=\begin{pmatrix}
   x_{11}   &  x_{12}  \\
    x_{21}  &  x_{22}
\end{pmatrix}\begin{pmatrix}
   c_1   &  0  \\
    0  &  c_2
\end{pmatrix}\begin{pmatrix}
   x_{11}   &  x_{12}  \\
    x_{21}  &  x_{22}
\end{pmatrix}-\begin{pmatrix}
   c_1   &  0  \\
    0  &  c_2
\end{pmatrix}\begin{pmatrix}
   x_{11}   &  x_{12}  \\
    x_{21}  &  x_{22}
\end{pmatrix}\begin{pmatrix}
   c_1   &  0  \\
    0  &  c_2
\end{pmatrix}=0.$$
A direct computation shows that these $x_{ij}$ satisfy the following quadratic equations:
\begin{equation} \label{eq3.1}
\begin{aligned}
c_1x_{11}x_{12} -c_1c_2x_{12}  + c_2x_{12}x_{22}&=0 
&c_1x_{11}x_{21}  -c_1c_2x_{21} + c_2x_{21}x_{22}&=0 \\
 c_1x_{12}x_{21} + c_2x_{22}^2 -c_2^2x_{22}&=0 
 &c_1x_{11}^2 -c_1^2x_{11}  + c_2x_{12}x_{21}&=0.
\end{aligned}
\end{equation}

\subsection{Exactly one of $\{c_1,c_2\}$ is zero} 
For this case, without loss of generality, we may assume that $c_2=0$ and $c_1\neq 0$. Then the four equations in (\ref{eq3.1}) can be reduced to:
\begin{equation} \label{eq3.2}
\begin{aligned}
x_{11}x_{12}&=0 &x_{11}x_{21}&=0 \\
x_{12}x_{21}&=0 &x_{11}^2 -c_1x_{11}&=0.
\end{aligned}
\end{equation}

An elementary method can compute the cardinality of $\D_{A_1}(\F_q)$ in this case. 

\begin{prop}\label{prop3.1}
If exactly one of $\{c_1,c_2\}$ is zero, then $|\D_{A_1}(\F_q)|=2q^2$.
\end{prop}

\begin{proof}
If $x_{11}=0$, then it follows from (\ref{eq3.2}) that either $x_{12}=0$ or $x_{21}=0$. This provides two types of solutions in
$\D_{A_1}(\F_q)$:  
$$D_1:=\left\{\begin{pmatrix}
   0   & 0   \\
   a   &  b
\end{pmatrix}~\Big|~ a,b\in\F_q\right\}\textrm{ and }D_2:=\left\{\begin{pmatrix}
    0  &  a \\
     0 &  b
\end{pmatrix}~\Big|~ a,b\in\F_q\right\}$$
respectively. Clearly, $|D_1|=q^2$ and $|D_2\setminus D_1|=q(q-1)$. If $x_{11}\neq 0$, then (\ref{eq3.2}) implies that
$x_{12}=x_{21}=0$ and $x_{11}=c_1$. We obtain the third type of solution in $\D_{A_1}(\F_q)$ as follows:
$$D_3:=\left\{\begin{pmatrix}
    c_1  &  0  \\
    0  & a
\end{pmatrix}~\Big|~ a\in\F_q\right\}.$$
Note that $|D_3|=q$. Since $\D_{A_1}(\F_q)=D_1\cup D_2\cup D_3$ and $D_3\cap (D_2\cup D_1)=\emptyset$, we see that
$$|\D_{A_1}(\F_q)|=|D_1|+|D_2\setminus D_1|+|D_3\setminus(D_2\cup D_1)|=q^2+q(q-1)+q$$
which simplifies to $2q^2$, as desired.
\end{proof}

\begin{prop}\label{prop3.2}
Assume that exactly one of $\{c_1,c_2\}$ is zero. Then $\D_{A_1}(\F_q)$ is a 2-dimensional affine variety consisting of 
two 2-dimensional irreducible components and one 1-dimensional irreducible component. In particular, the vanishing ideal
of $\D_{A_1}(\F_q)$ is generated by $$\{x_{11}x_{12}, x_{11}x_{21}, x_{12}x_{21}, x_{11}^2 -c_1x_{11}\}.$$
\end{prop}

\begin{proof}
We will be working in the polynomial ring $R:=\F_q[x_{11},x_{12},x_{21},x_{22}]$. Suppose that $\p_1=\ideal{x_{11},x_{12}}$, $\p_2=\ideal{x_{11},x_{21}}$, and $\p_3=\ideal{x_{11}-c_1, x_{12},x_{21}}$.  Clearly, they are prime ideals in $R$ because $$R/\p_1\cong \F_q[x_{21},x_{22}],~ R/\p_2\cong \F_q[x_{12},x_{22}],~ R/\p_3\cong \F_q[x_{22}]$$ are integral domains. Moreover, $\p_1,\p_2$ and $\p_3$ are the vanishing ideals corresponding to $D_1,D_2$, and $D_3$ in Proposition \ref{prop3.1}, respectively. Thus each $D_i=V(\p_i)$ is irreducible and $$\D_{A_1}(\F_q)=D_1\cup D_2\cup D_3=V(\p_1)\cup V(\p_2)\cup V(\p_3).$$ Note that the Krull dimensions of $D_1$ and $D_2$ both are 2 and the Krull dimension of $D_3$ is 1. Hence,
the first statement follows.

To show the second statement, let $J$ denote the ideal generated by 
$\{x_{11}x_{12}, x_{11}x_{21}, x_{12}x_{21}, x_{11}^2 -c_1x_{11}\}$ in $R$.  By (\ref{eq3.2}), we see that $\D_{A_1}(\F_q)=V(J)$ and so the vanishing ideal $I(\D_{A_1}(\F_q))=\sqrt{J}$. Thus, we only need to show that
$\sqrt{J}\subseteq J$. By \cite[Lemma 1.1.5]{Gec03}, we see that
$$\sqrt{J}=I(\D_{A_1}(\F_q))=I(V(\p_1)\cup V(\p_2)\cup V(\p_3))=I(V(\p_1\cap\p_2\cap\p_3))=I(V(\p_1\cdot\p_2\cdot\p_3))=\sqrt{\p_1\cdot\p_2\cdot\p_3}.$$
As all $\p_i$ are prime ideals, it follows that $\sqrt{\p_1\cdot\p_2\cdot\p_3}=\p_1\cdot\p_2\cdot\p_3$.
Thus, it suffices to show that
\begin{equation}
\label{ }
\p_1\cdot\p_2\cdot\p_3\subseteq J.
\end{equation}
By \cite[Proposition 6, page 191]{CLO15}, which asserts that the product of two ideals in a polynomial ring can be generated by the product of generators of the ideals, we see that the ideal $\p_1\cdot\p_2$ can be generated by
$\{x_{11}^2, x_{11}x_{12}, x_{11}x_{21}, x_{12}x_{21}\}$ and thus $\p_1\cdot\p_2\cdot\p_3$ can be generated by
all products $f\cdot g$, where $f\in \{x_{11}^2, x_{11}x_{12}, x_{11}x_{21}, x_{12}x_{21}\}$ and $g\in \{x_{11}-c_1, x_{12},x_{21}\}$. A direct calculation shows that in the quotient ring $R/J$, the classes of all such products $f\cdot g$ will be zero. Thus 
$\p_1\cdot\p_2\cdot\p_3$ is contained in $J$. This proves that $J$ is the vanishing ideal
of $\D_{A_1}(\F_q)$ and furthermore, $J=\sqrt{J}=\p_1\cap\p_2\cap\p_3$.
\end{proof}

Combining Propositions \ref{prop3.1} and \ref{prop3.2} completes the proof of the first statement in Theorem \ref{mt1}.

\subsection{$c_1\neq c_2\in\F_q^\times$}
To give a proof to the second statement in Theorem \ref{mt1}, we define
$$\updelta:=c_1^2  -c_1c_2+ c_2^2.$$
The geometric structure of $\D_{A_1}(\F_q)$ differs depending on whether  $\updelta =0$ or $\updelta\neq0$.

\begin{prop}\label{prop3.3}
Let $c_1$ and $c_2$ be two distinct nonzero elements in $\F_q$. 
\begin{enumerate}
  \item Suppose that $\updelta=0$. Then $|\D_{A_1}(\F_q)|=2q+2$ and the variety $\D_{A_1}(\F_q)$ is a 1-dimensional affine variety consisting of three isolated points 
$$\left\{\begin{pmatrix}
     0 & 0   \\
    0  &  0
\end{pmatrix},\begin{pmatrix}
   c_1  &  0  \\
   0   & 0 
\end{pmatrix}, \begin{pmatrix}
   0   &  0  \\
   0   & c_2 
\end{pmatrix}\right\}$$
and two 1-dimensional irreducible components
$$E_1:=\left\{\begin{pmatrix}
     c_1 & 0   \\
      a& c_2 
\end{pmatrix}~\Big|~ a\in\F_q\right\}\textrm{ and } E_2:=\left\{\begin{pmatrix}
     c_1 & a   \\
      0& c_2 
\end{pmatrix}~\Big|~ a\in\F_q\right\}.$$
  \item Suppose that $\updelta\neq 0$. Then $|\D_{A_1}(\F_q)|=q+3$ and the variety $\D_{A_1}(\F_q)$ is also a 1-dimensional affine variety consisting of four isolated points 
$$\left\{\begin{pmatrix}
     0 & 0   \\
    0  &  0
\end{pmatrix},\begin{pmatrix}
   c_1  &  0  \\
   0   & 0 
\end{pmatrix},\begin{pmatrix}
   0   &  0  \\
   0   & c_2 
\end{pmatrix},\begin{pmatrix}
     c_1 & 0   \\
      0& c_2 
\end{pmatrix} \right\}$$
and one 1-dimensional irreducible component
$$E_3:=\left\{\begin{pmatrix}
     \frac{c_2^2}{c_2-c_1} & a   \\
      b& \frac{c_1^2}{c_1-c_2}
\end{pmatrix}~\Big|~ a,b\in\F_q^\times\textrm{ and }ab=-\frac{c_1c_2\cdot \updelta}{(c_1-c_2)^2}\right\}.$$
\end{enumerate}
\end{prop}

\begin{proof}
Our proof will be separated into two cases:  $x_{11}x_{12}=0$ and  $x_{11}x_{12}\neq0$.  

\textsc{Case 1}: Assume that $x_{11}x_{12}=0$. We proceed with the following three subcases: (1) $x_{11}=x_{12}=0$. In this subcase,  the equations in (\ref{eq3.1}) can be reduced to
\begin{eqnarray*}
-c_1x_{21} + x_{21}x_{22}= 0 &~~& x_{22}^2 -c_2x_{22} =0.
\end{eqnarray*}
By these two equations, we see that if $x_{22}=0$, then $x_{21}=0$ and we have $X=0.$ Moreover, if $x_{22}\neq 0$, then
$x_{22}=c_2$. Hence, $(c_2-c_1)x_{21}=0$,
which implies that $x_{21}=0$. This produces an isolated nonzero solution in $\D_{A_1}$:
$$X=\begin{pmatrix}
   0   &  0  \\
   0   & c_2 
\end{pmatrix}.$$
(2) $x_{11}=0$ and $x_{12}\neq 0$.  We may simplify (\ref{eq3.1}) as follows:
\begin{equation*} 
\begin{aligned}
 -c_1 + x_{22}&=0 & -c_1x_{21} + x_{21}x_{22}&= 0 \\
 c_1x_{12}x_{21} + c_2x_{22}^2 -c_2^2x_{22} &=0& x_{21}&=0.
\end{aligned}
\end{equation*}
Then $x_{22}=c_1$ and $c_2c_1^2=c_2^2c_1$. Since $c_1$ and $c_2$ both are nonzero, it follows that $c_1=c_2$, which contradicts with the original assumption that $c_1\neq c_2$. This means that
if $x_{11}=0$, then $x_{12}=0$. There are no new solutions obtained in this subcase. (3) $x_{11}\neq0$ and $x_{12}=0$.  We can reduce (\ref{eq3.1}) into the following equations:
\begin{eqnarray*}
c_1x_{11}x_{21}  -c_1c_2x_{21} + c_2x_{21}x_{22}=0 &~~
c_2x_{22}^2 -c_2^2x_{22} =0 &~~ c_1x_{11}^2 -c_1^2x_{11} =0.
\end{eqnarray*}
By the third and second equations, we see that $x_{11}=c_1$ and $x_{22}=c_2$ respectively. Substituting back to the first equation gives rise to:
\begin{equation}
\label{ }
c_1^2x_{21}  -c_1c_2x_{21} + c_2^2x_{21}=0.
\end{equation}
Recall that  $\updelta=c_1^2  -c_1c_2+ c_2^2.$
If $\updelta=0$, then $x_{21}$ is a free variable and 
$$X=\begin{pmatrix}
     c_1 & 0   \\
      a& c_2 
\end{pmatrix}\in E_1.$$
If $\updelta\neq0$, then $x_{21}=0$ and we obtain the following isolated solution: 
$$X=\begin{pmatrix}
     c_1 & 0   \\
      0& c_2 
\end{pmatrix}.$$

\textsc{Case 2}: Assume that $x_{11}x_{12}\neq 0$. In this case, the equations in (\ref{eq3.1}) can be reduced to
\begin{eqnarray}\label{eq3.5}
\begin{aligned}
c_1x_{11}-c_1c_2 + c_2x_{22}& =  0 \\
 c_1x_{12}x_{21} + c_2x_{22}^2 -c_2^2x_{22} &=0\\
   c_1x_{11}^2 -c_1^2x_{11}  + c_2x_{12}x_{21}&=0.
   \end{aligned}
\end{eqnarray}
(1) If $x_{21}=0$, then (\ref{eq3.5}) becomes
\begin{eqnarray*}
c_1x_{11}-c_1c_2 + c_2x_{22}= 0
&c_2x_{22}^2 -c_2^2x_{22} =0&
   c_1x_{11}^2 -c_1^2x_{11} =0.
\end{eqnarray*}
This implies that $x_{11}=c_1$ and $x_{22}=c_2$. Thus $\updelta=0$
and $x_{12}$ is a free variable. We obtain the second family $E_2$ of solutions: 
$$X=\begin{pmatrix}
     c_1 &  a  \\
     0 &  c_2
\end{pmatrix}.$$ 
(2) If $x_{21}\neq0$, then it follows from the first equation in (\ref{eq3.5}) that $c_2x_{22}=c_1c_2-c_1x_{11}.$
Together with the second and third equations in (\ref{eq3.5}), we may cancel the term involving $x_{12}x_{21}$ and derive
\begin{eqnarray*}
0&=&c_2^2x_{22}^2 -c_2^3x_{22}-   c_1^2x_{11}^2 +c_1^3x_{11} \\
&=&(c_1c_2-c_1x_{11})^2-c_2^2(c_1c_2-c_1x_{11})-   c_1^2x_{11}^2 +c_1^3x_{11}\\
&=&c_1^2c_2^2-2c_1^2c_2x_{11}+c_1^2x_{11}^2-c_1c_2^3+c_1c_2^2x_{11}-   c_1^2x_{11}^2 +c_1^3x_{11}\\
&=& (c_1^3-2c_1^2c_2+c_1c_2^2)x_{11}+(c_1^2c_2^2-c_1c_2^3).
\end{eqnarray*}
Hence,
$$x_{11}=\frac{c_1c_2^3-c_1^2c_2^2}{c_1^3-2c_1^2c_2+c_1c_2^2}=\frac{c_2^2(c_2-c_1)}{c_1^2-2c_1c_2+c_2^2}=\frac{c_2^2}{c_2-c_1}.$$
Moreover, 
$$x_{22}=c_2^{-1}(c_1c_2-c_1x_{11})=\frac{c_1^2}{c_1-c_2}$$
and it follows from the third equation in (\ref{eq3.5}) that
$$x_{12}x_{21}=-\frac{c_1c_2\cdot \updelta}{(c_1-c_2)^2}.$$
If $\updelta=0$, then there are no new solutions produced because of $x_{12}x_{21}=0$ and the solutions either belong to
$E_1$ or $E_2$. If $\updelta\neq 0$, then we obtain the third family of solutions:
$$X=\begin{pmatrix}
     \frac{c_2^2}{c_2-c_1} & a   \\
      b& \frac{c_1^2}{c_1-c_2}
\end{pmatrix}\in E_3$$ 
where $ab=-\frac{c_1c_2\cdot \updelta}{(c_1-c_2)^2}\neq 0$.

Summarizing these results, we conclude that when $\updelta=0$, the cardinality of $\D_{A_1}(\F_q)$ equals $2q+2$ and $\D_{A_1}(\F_q)$  consists of three isolated points and two 1-dimensional irreducible components $E_1$ and $E_2$; thus it is
a 1-dimensional affine variety. When $\updelta\neq0$, $|\D_{A_1}(\F_q)|=q+3$ and $\D_{A_1}(\F_q)$  consists of four isolated points and one 1-dimensional  irreducible component $E_3$.  
\end{proof}

Combining Propositions \ref{prop3.1}, \ref{prop3.2} and \ref{prop3.3}  completes the proof of Theorem \ref{mt1}.

\section{Geometry of $\D_{A_2}(\F_q)$} \label{sec4}
\setcounter{equation}{0}
\renewcommand{\theequation}
{4.\arabic{equation}}
\setcounter{theorem}{0}
\renewcommand{\thetheorem}
{4.\arabic{theorem}}

\noindent Throughout  this section, we always assume that $A_2=\begin{pmatrix}
   c   &  1  \\
    0  &  c
\end{pmatrix}$
for some $c\in\F_q$. Consider a generic solution $X=\begin{pmatrix}
   x_{11}   &  x_{12}  \\
    x_{21}  &  x_{22}
\end{pmatrix}\in\D_{A_2}(\F_q)$. It follows from (\ref{YBME}) that
$$XA_2X-A_2XA_2=\begin{pmatrix}
   x_{11}   &  x_{12}  \\
    x_{21}  &  x_{22}
\end{pmatrix}\begin{pmatrix}
   c   &  1  \\
    0  &  c
\end{pmatrix}\begin{pmatrix}
   x_{11}   &  x_{12}  \\
    x_{21}  &  x_{22}
\end{pmatrix}-\begin{pmatrix}
   c   &  1  \\
    0  &  c
\end{pmatrix}\begin{pmatrix}
   x_{11}   &  x_{12}  \\
    x_{21}  &  x_{22}
\end{pmatrix}\begin{pmatrix}
   c   &  1  \\
    0  &  c
\end{pmatrix}=0.$$
A straightforward calculation implies  that
\begin{equation} \label{eq4.1}
\begin{aligned}
-c^2x_{12} + cx_{11}x_{12} -cx_{11} + cx_{12}x_{22} -cx_{22} + x_{11}x_{22} -x_{21}&=0\\
    -c^2x_{11} + cx_{11}^2 + cx_{12}x_{21} -cx_{21} + x_{11}x_{21}&=0\\
    -c^2x_{21} + cx_{11}x_{21} + cx_{21}x_{22} + x_{21}^2&=0\\
    -c^2x_{22} + cx_{12}x_{21} -cx_{21} + cx_{22}^2 + x_{21}x_{22}&=0.
 \end{aligned}
\end{equation}
This section is devoted to giving a proof to Theorem \ref{mt2} and our argument will be divided into two cases: $c=0$ and $c\neq 0$.

\subsection{$c=0$}  
The following result provides a proof to the first statement in Theorem \ref{mt2}.

\begin{prop}\label{prop4.1}
If $c=0$, then $\D_{A_2}(\F_q)$ is of cardinality $2q^2-q$ and is a 2-dimensional affine variety consisting of 
two 2-dimensional irreducible components. 
\end{prop}

\begin{proof}
As $c=0$, the equations in (\ref{eq4.1}) become to
\begin{equation}
\label{eq4.2}
x_{11}x_{22} -x_{21}=0,~~\quad x_{11}x_{21}=0,~~\quad x_{21}^2=0,~~\quad x_{21}x_{22}=0.
\end{equation}
The fact that $x_{21}^2=0$ implies that $x_{21}=0$. Thus the second and the fourth equations in (\ref{eq4.2}) disappear.  
The first equation shows that either $x_{11}=0$ or $x_{22}=0$. If $x_{11}=0$, then $x_{22}$ and $x_{12}$ are two free variables and we obtain the first family of solutions in $\D_{A_2}(\F_q)$:
$$X=\begin{pmatrix}
      0& a   \\
     0 &  b
\end{pmatrix}.$$
If $x_{11}\neq 0$, then $x_{22}$ must be zero. Thus $x_{11}$ and $x_{12}$ are free variables and this gives us the second family of solutions in $\D_{A_2}(\F_q)$:
$$X=\begin{pmatrix}
      a& b \\
     0 &  0
\end{pmatrix}.$$ Here, $a,b$ are any elements in $\F_q$.  Clearly, the vanishing ideal of the first family is
$\p_1:=\ideal{x_{11}, x_{21}}$ and the vanishing ideal of the second family is $\p_2:=\ideal{x_{21}, x_{22}}$. Note that the two ideals both are prime, thus 
$$V(\p_1)=\left\{\begin{pmatrix}
      0& a   \\
     0 &  b
\end{pmatrix}~\Big|~ a,b\in \F_q\right\}\textrm{ and } V(\p_2)=\left\{\begin{pmatrix}
      a& b   \\
     0 &  0
\end{pmatrix}~\Big|~ a,b\in \F_q\right\}$$
both are 2-dimensional irreducible components of $\D_{A_2}(\F_q)$. Moreover, 
$$|\D_{A_2}(\F_q)|=|V(\p_1)\cup V(\p_2)|=|V(\p_1)|+|V(\p_2)|-|V(\p_1)\cap V(\p_2)|=q^2+q^2-q=2q^2-q$$
because $|V(\p_1)|=|V(\p_2)|=q^2$ and $|V(\p_1)\cap V(\p_2)|=q.$
\end{proof}

\subsection{$c\neq0$} 
In this subsection, we assume that $c\neq 0$. The following result proves the second statement in Theorem \ref{mt2}.

\begin{prop}\label{prop4.2}
If $c\neq 0$, then $\D_{A_2}(\F_q)$  is a 1-dimensional affine variety consisting of two isolated points and 
one 1-dimensional irreducible component. In particular, $|\D_{A_2}(\F_q)|=q+2$.
\end{prop}

\begin{proof}
Our argument will be separated into two subcases: $x_{21}=0$ and $x_{21}\neq0$. First of all, we assume that $x_{21}=0$. This assumption makes (\ref{eq4.1}) become
\begin{equation} \label{eq4.3}
\begin{aligned}
c^2x_{12} - cx_{11}x_{12} +cx_{11} - cx_{12}x_{22} + cx_{22} - x_{11}x_{22}&=0\\
    c^2x_{11} - cx_{11}^2 =0,~~\quad c^2x_{22}  - cx_{22}^2 &=0.
 \end{aligned}
\end{equation}
Note that $c\neq 0$, the last two equations in (\ref{eq4.3}) imply that $x_{11},x_{22}\in\{0,c\}$. Thus, we have four possibilities for the values of $x_{11}$ and $x_{22}$. If $(x_{11},x_{22})=(0,0)$, then substituting it to the first equation in  
(\ref{eq4.3}) gives $x_{12}=0$ and so we obtain the zero solution in $\D_{A_2}(\F_q)$.  If $(x_{11},x_{22})=(0,c)$ or $(x_{11},x_{22})=(c,0)$, then it follows from the first equation in (\ref{eq4.3}) that $c=0$, which is a contradiction. If $(x_{11},x_{22})=(c,c)$, then $x_{12}=1$ by (\ref{eq4.3}). Thus, the following matrix
$$X=\begin{pmatrix}
      c& 1   \\
     0 &  c
\end{pmatrix}=A_2$$
was revealed as a point in $\D_{A_2}(\F_q)$.

Now, let us assume that $x_{21}\neq0$. Dividing the third equation in (\ref{eq4.1}) by $x_{21}$, we see that
\begin{eqnarray} 
-c^2x_{12} + cx_{11}x_{12} -cx_{11} + cx_{12}x_{22} -cx_{22} + x_{11}x_{22} -x_{21}&=&0\label{eq4.4}\\
    -c^2x_{11} + cx_{11}^2 + cx_{12}x_{21} -cx_{21} + x_{11}x_{21}&=&0\label{eq4.5}\\
    -c^2+ cx_{11} + cx_{22} + x_{21}&=&0\label{eq4.6}\\
    -c^2x_{22} + cx_{12}x_{21} -cx_{21} + cx_{22}^2 + x_{21}x_{22}&=&0.\label{eq4.7}
\end{eqnarray}
Adding (\ref{eq4.6}) on (\ref{eq4.4}) obtains 
\begin{equation}
\label{eq4.8}
-c^2x_{12} + cx_{11}x_{12}  + cx_{12}x_{22}  + x_{11}x_{22}   -c^2  =0.
\end{equation}
It follows from (\ref{eq4.6}) that  $cx_{11} + cx_{22}=c^2-x_{21}$. Substituting this fact to (\ref{eq4.8}) gives
\begin{eqnarray*}
c^2&=&x_{11}x_{22}-c^2x_{12} + (cx_{11} + cx_{22})x_{12} \\
&=&x_{11}x_{22}-c^2x_{12} +(c^2-x_{21})x_{12}\\
&=&x_{11}x_{22}-x_{12}x_{21}.
\end{eqnarray*}
By (\ref{eq4.7}) and (\ref{eq4.5}) , we observe that $x_{12}x_{21}  =  cx_{22} +x_{21} - x_{22}^2 - c^{-1}x_{21}x_{22}$ and 
$x_{11}=c- c^{-1}x_{21} - x_{22}$. Thus
$c^2= x_{11}x_{22}-x_{12}x_{21}=(cx_{22}- c^{-1}x_{21}x_{22} - x_{22}^2)- (cx_{22} +x_{21} - x_{22}^2 - c^{-1}x_{21}x_{22})=-x_{21}$. This implies that
\begin{equation}
\label{ }
x_{21}=-c^2.
\end{equation}
Moreover, $x_{11}=c- c^{-1}x_{21} - x_{22}=c- c^{-1}(-c^2) - x_{22}=2c-x_{22}$. We substitute $x_{21}$ and $x_{11}$ back to (\ref{eq4.7}) and obtain 
\begin{equation}
\label{ }
 - c^2x_{12} +c^2 + x_{22}^2 -2cx_{22}=0,
\end{equation}
which means that $x_{12}=c^{-2}x_{22}^2 -2c^{-1}x_{22}+1.$ Hence, we deduce the following family of solutions in $\D_{A_2}(\F_q)$:
$$E:=\left\{\begin{pmatrix}
   2c-a   &  c^{-2}a^2 -2c^{-1}a+1  \\
 -c^2     &  a
\end{pmatrix} ~\Big|~ a\in\F_q\right\}.$$
Therefore, 
$$\D_{A_2}(\F_q)=\left\{\begin{pmatrix}
   0   &  0  \\
    0  & 0 
\end{pmatrix},\begin{pmatrix}
      c& 1   \\
     0 &  c
\end{pmatrix}\right\}\cup E$$
and $|\D_{A_2}(\F_q)|=q+2$. 

Note that the Krull dimension of $E$ is 1. To complete the proof, we need to prove that $E$ is an irreducible component of 
$\D_{A_2}(\F_q)$, i.e., the vanishing ideal of $E$ is a prime ideal in $R$. Let $\p$ be the ideal generated by
$\{x_{21}+c^2,x_{11}+x_{22}-2c, x_{12}-c^{-2}x_{22}^2 +2c^{-1}x_{22}-1 \}$ in $R$. Then the radical of $\p$ equals the vanishing ideal of $E$. Thus, it suffices to show that $\p$ is prime, which is equivalent to verifying that 
$R/\p$ is an integral domain. In fact, 
$$R/\p=\F_q[x_{11},x_{12},x_{21},x_{22}]/\ideal{x_{21}+c^2,x_{11}+x_{22}-2c, x_{12}-c^{-2}x_{22}^2 +2c^{-1}x_{22}-1}\cong \F_q[x_{22}]$$
is an integral domain. Hence, $\p$ is a prime ideal and $E$ is irreducible. 
\end{proof}

\begin{rem}{\rm 
Combining Propositions \ref{prop4.1} and \ref{prop4.2}  completes the proof of Theorem \ref{mt2}.
\hbo}\end{rem}

\section{Geometry of $\D_{A_3}(\F_q)$} \label{sec5}
\setcounter{equation}{0}
\renewcommand{\theequation}
{5.\arabic{equation}}
\setcounter{theorem}{0}
\renewcommand{\thetheorem}
{5.\arabic{theorem}}

\noindent In this section, we assume that $\cha(\F_q)\neq 2$ and study the geometry of the affine variety $\D_{A_3}(\F_q)$, where $A_3=\begin{pmatrix}
   0   &  -b  \\
    1  &  a
\end{pmatrix}$
and $a,b\in\F_q$ such that $x^2-ax+b$ is irreducible in $\F_q[x]$. We write 
$$\Delta:=a^2-4b$$
for the discriminant of a quadratic function $x^2-ax+b$ over any field $\K$. 

The following lemma might be well-known  but
we include a proof to this result for the readers' convenience, because we were unable to find a suitable reference.

\begin{lem}\label{lem5.1}
If $\cha(\K)\neq 2$, then $x^2-ax+b$ irreducible if and only if $\Delta$ is a non-square in $\K$.
\end{lem}

\begin{proof}
$(\RA)$ Assume by the way of contradiction that there exists an element $c\in\K$ such that $\Delta=c^2$. Then $\frac{a+c}{2}$ and $\frac{a-c}{2}$ both belong to $\K$. Moreover, 
\begin{eqnarray*}
\left(x-\frac{a+c}{2}\right)\left(x-\frac{a-c}{2}\right)&=&x^2-\left(\frac{a+c}{2}+\frac{a-c}{2}\right)x+\frac{a+c}{2}\cdot\frac{a-c}{2} \\
&=&x^2-ax+\frac{a^2-\Delta}{4}\\
&=&x^2-ax+b
\end{eqnarray*}
which means that $x^2-ax+b$ is reducible over $\K$. This contradiction shows that if  $x^2-ax+b$ irreducible, then
$\Delta$ is a non-square in $\K$. 

$(\LA)$ Conversely, we assume that $x^2-ax+b$ is reducible, i.e., there are two elements
$r,s\in\K$ such that $x^2-ax+b=(x-r)(x-s)=x^2-(r+s)x+rs$. Thus
$a=r+s$ and $b=rs$. This implies that $\Delta=a^2-4b=(r+s)^2-4rs=(r-s)^2$ is a square element in $\K$. This contradiction shows that $x^2-ax+b$ is irreducible over $\K$.
\end{proof}

\begin{rem}{\rm 
By Lemma \ref{lem5.1}, we see that $b\neq 0$ because if $b=0$, then $\Delta=a^2$ is a square in $\F_q$. 
\hbo}\end{rem}

The main purpose of this section is to prove Theorem \ref{mt3}, in which we assume that 
\begin{equation}
\label{ }
\Delta= -b.
\end{equation}
Thus $a^2=3b$ and $a,b$ are not zero.

We begin with a generic solution $X=\begin{pmatrix}
   x_{11}   &  x_{12}  \\
    x_{21}  &  x_{22}
\end{pmatrix}\in\D_{A_3}(\F_q)$. Thus $XA_3X-A_3XA_3=0$
gives rise to the following equations:
\begin{equation} \label{eq5.2}
\begin{aligned}
ax_{12}x_{21} -bx_{11}x_{21} + bx_{22} + x_{11}x_{12}&=0\\
    ax_{21}x_{22} -ax_{22} -bx_{21}^2 + x_{11}x_{22} -x_{12} &=0\\
    abx_{22} + ax_{12}x_{22} -b^2x_{21} -bx_{11}x_{22} + x_{12}^2 &=0\\
    a^2x_{22} - abx_{21} +ax_{12} - ax_{22}^2 - bx_{11} +bx_{21}x_{22} - x_{12}x_{22}&=0.
 \end{aligned}
\end{equation}
Let us write $f_1,\dots,f_4$ for the quadratic functions corresponding to the  equations in (\ref{eq5.2}) and define 
$$J:=\ideal{f_1,\dots,f_4}.$$
We also define $\p_1$ to be the ideal generated by
$$\left\{g_1:=x_{11} + x_{22} - a, g_2:=x_{12} -bx_{21} +ax_{22} - b, g_3:=x_{21}^2 -\frac{a}{b}x_{21}x_{22} + x_{21} +\frac{1}{b}x_{22}^2 -\frac{a}{b}x_{22} + 1\right\}$$
and $\p_2$ to be the ideal generated by
$$\left\{h_1:=x_{11} + 2x_{22} -a, h_2:=x_{12} + ax_{22} -b, h_3:=x_{21} -\frac{a}{b}x_{22} + 1,
h_4:=x_{22}^2 -ax_{22} + b \right\}.$$
Moreover, we define $\p_3:=\ideal{x_{11},x_{12},x_{21},x_{22}}$.

\begin{lem}\label{lem5.3}
$\p_1,\p_2,\p_3$ are prime ideals.
\end{lem}

\begin{proof}
Clearly, $\p_3$ is prime because $R/\p_3\cong \F_q$ is a field and so an integral domain. To see that $\p_2$ is prime, we note that $R/\p_2\cong \F_q[x_{22}]/\ideal{h_4}$  and together with the fact that $h_4=x_{22}^2 -ax_{22} + b$ is irreducible, it follows that $\F_q[x_{22}]/\ideal{h_4}$ is an integral domain. Thus, $\p_2$ is also a prime ideal. To show that $\p_1$ is prime, we observe that $R/\p_1\cong \F_{q}[x_{21},x_{22}]/\ideal{g_3}$ and the latter can be embedded into the quotient ring
$\F_q(x_{22})[x_{21}]/\ideal{g_3}$. Thus, it suffices to show that $\F_q(x_{22})[x_{21}]/\ideal{g_3}$ is an integral domain. 
Now we may review 
$$g_3=x_{21}^2 -\left(\frac{a}{b}x_{22} -1\right)x_{21} +\left(\frac{1}{b}x_{22}^2 -\frac{a}{b}x_{22} + 1\right)$$
as a polynomial over the rational function field $\F_q(x_{22})$ in the variable $x_{21}$. Note that $\Delta=-b\neq 0$ and  $a^2=3b$. The discriminant of $g_3$ is
\begin{eqnarray*}
& & \left(\frac{a}{b}x_{22} -1\right)^2-4\left(\frac{1}{b}x_{22}^2 -\frac{a}{b}x_{22} + 1\right) \\
 & = & \frac{a^2}{b^2}x_{22}^2-  \frac{2a}{b}x_{22}+1-\frac{4}{b}x_{22}^2 +\frac{4a}{b}x_{22} -4\\
 &=&\frac{a^2-4b}{b^2}x_{22}^2+\frac{2a}{b}x_{22}-3\\
  &=&\frac{-b}{b^2}x_{22}^2+\frac{2a}{b}x_{22}-\frac{3b}{b}\\
  &=&-\frac{1}{b}(x_{22}^2-2ax_{22}+a^2)=\frac{(x_{22}-a)^2}{\Delta}.
\end{eqnarray*}
Since $\Delta$ is a non-square element, it follows that the discriminant of $g_3$ is also a non-square element. 
By Lemma \ref{lem5.1}, we see that $g_3$ is irreducible over $\F_q(x_{22})$. Hence, $\F_q(x_{22})[x_{21}]/\ideal{g_3}$ is an integral domain and so is  $R/\p_1$. This proves that $\p_1$ is also a prime ideal.
\end{proof}

\begin{lem}\label{lem5.4}
$\sqrt{J}=\p_1\cap\p_2\cap\p_3$.
\end{lem}

\begin{proof}
It is straightforward to verify that  $f_1,\dots,f_4$ belong to each $\p_i$. For instance, to see that $f_1\in\p_2$, we only need to check that $f_1\equiv 0$ modulo $\p_2$. In $R/\p_2$, 
$x_{11}\equiv a-2x_{22}, x_{12}\equiv b-ax_{22}$, and $x_{21}\equiv\frac{a}{b}x_{22} - 1$. Thus
\begin{eqnarray*}
f_1 & = & a(b-ax_{22})\left(\frac{a}{b}x_{22} - 1\right) -b(a-2x_{22})\left(\frac{a}{b}x_{22} - 1\right) + bx_{22} + (a-2x_{22})(b-ax_{22}) \\
 & \equiv & \left(a^2x_{22}-ab-\frac{a^3}{b}x_{22}^2+a^2x_{22}\right)-(a^2x_{22}-ab-2ax_{22}^2+ 2bx_{22})+ bx_{22}\\
 && +(ab-a^2x_{22}-2bx_{22}+2ax_{22}^2)\\
 &=&-\frac{3ab}{b}x_{22}^2+a^2x_{22}+2ax_{22}^2- 2bx_{22}+ bx_{22}+ab-a^2x_{22}-2bx_{22}+2ax_{22}^2\\
 &=&ax_{22}^2-a^2x_{22}+ab=a\cdot h_4
\end{eqnarray*}
which is equivalent to $0$ modulo $\p_2$. These direct verifications imply that $J\subseteq \p_1\cap\p_2\cap\p_3$. 
By Lemma \ref{lem5.3}, we have seen that all $\p_i$ are prime ideals. Thus $\sqrt{J}\subseteq \p_1\cap\p_2\cap\p_3$ and 
$V(\p_1\cap\p_2\cap\p_3)\subseteq V(\sqrt{J})$.
By \cite[Lemma 1.1.5]{Gec03}, we have
$$\sqrt{J}=I(V(\sqrt{J}))\subseteq I(V(\p_1\cap\p_2\cap\p_3))=I(V(\p_1\cdot\p_2\cdot\p_3))=\sqrt{\p_1\cdot\p_2\cdot\p_3}.$$
It follows from Lemma \ref{lem5.3} that $\sqrt{\p_1\cdot\p_2\cdot\p_3}=\p_1\cdot\p_2\cdot\p_3$. Thus, it suffices to show that $\p_1\cdot\p_2\cdot\p_3\subseteq J.$ By \cite[Proposition 6, page 191]{CLO15}, the set $\A:=\{g_k\cdot x_{ij}\mid 1\leqslant k\leqslant 3, 1\leqslant i,j\leqslant 2\}$ can generate $\p_1\cdot \p_3$ and the set $\B:=\{a\cdot h_i\mid a\in\A,1\leqslant i\leqslant 4\}$
could be a generating set of $\p_1\cdot\p_2\cdot\p_3$. 

To show that the class of each element in $\B$ in $R/J$ is zero, we need to show that each element in $\B$ belongs to $J$.
To do this, we may use a method of Gr\"{o}bner basis. Let us choose the lexicographic monomial ordering with $x_{11}>x_{12}>x_{21}>x_{22}$ for $R$. Note that $a^2=3b$ and $a,b$ are not zero. Using the Buchberger's algorithm (see for example, \cite[Algorithm 1.1.9]{DK15}), we may find a 
Gr\"{o}bner basis $\{s_1,s_2,\dots,s_6\}$ for $J$, where
\begin{eqnarray*}
s_1 & := & x_{11} + \frac{1}{b}x_{12}x_{22} - \frac{a}{b}x_{12} - x_{21}x_{22} + ax_{21} + \frac{a}{b}x_{22}^2 - \frac{a^2}{b}x22 \\
s_2 & := &  x_{12}^2 + ax_{12}x_{22} - bx_{12} - b^2x_{21}^2 + abx_{21}x_{22} - b^2x_{21}\\
s_3 & := & x_{12}x_{21} - \frac{a}{b}x_{12}x_{22} + x_{12} - \frac{2b}{a}x_{21}^2x_{22} + 2x_{21}x_{22}^2 + \frac{b}{a}x_{21}x_{22} -bx_{21} - \frac{2}{a}x_{22}^3 + \frac{b}{a}x_{22}\\
s_4 & := & x_{12}x_{22}^2 - ax_{12}x_{22} + bx_{12} + b^2x_{21}^2 - bx_{21}x_{22}^2 + ax_{22}^3 - a^2x_{22}^2 + abx_{22}\\
s_5 & := & x_{21}^3 - \frac{4}{a}x_{21}^2x_{22} + x_{21}^2 + \frac{2}{b}x_{21}x_{22}^2 - \frac{4}{a}x_{21}x_{22} + x_{21} - 
        \frac{1}{ab}x_{22}^3 + \frac{1}{b}x_{22}^2 - \frac{1}{a}x_{22}\\
s_6 & := &   x_{21}^2x_{22}^2 - ax_{21}^2x_{22} - \frac{a}{b}x_{21}x_{22}^3 + 4x_{21}x_{22}^2 - ax_{21}x_{22} + 
        \frac{1}{b}x_{22}^4 - \frac{2a}{b}x_{22}^3 + 4x_{22}^2 - ax_{22}.     
\end{eqnarray*}
By \cite[Algorithm 1.1.7]{DK15}, it suffices to verify that the normal form of each element in $\B$ with respect to 
$\{s_1,s_2,\dots,s_6\}$ is zero, which can be done by directly applying \cite[Algorithm 1.1.6]{DK15}.  A MAGMA code for computing Gr\"{o}bner bases $\{s_1,s_2,\dots,s_6\}$ of $J$ and the corresponding normal forms of elements in $\B$ can be found in Appendix.

Therefore, $\p_1\cdot\p_2\cdot\p_3$ is contained in $J$, and $\sqrt{J}=\p_1\cap\p_2\cap\p_3$.
\end{proof}

We are ready to complete the proof of Theorem \ref{mt3}.

\begin{proof}[Proof of Theorem \ref{mt3}]
By Lemma \ref{lem5.4},  $\D_{A_3}(\F_q)=V(\sqrt{J})=V(\p_1\cap\p_2\cap\p_3)=V(\p_1)\cup V(\p_2)\cup V(\p_3).$ Clearly, $V(\p_3)=\{0\}$. Moreover, $V(\p_2)=\emptyset$. To see that, $x_{11},x_{12}$, and $x_{21}$ are all determined by $x_{22}$ via $h_1,f_2,f_3$ respectively, but $h_4=x_{22}^2 -ax_{22} + b$ is irreducible over $\F_q$, thus there are no $x_{22}\in\F_q$
such that $h_4=0$. In other words, $V(\p_2)=\emptyset$. 

Let us look at $V(\p_1)$. As $x_{11}$ and $x_{12}$ are determined by $x_{21}$ and $x_{22}$ via $g_1$ and $g_2$, we only need to focus on the equation $g_3=0$, i.e.,
\begin{equation}
\label{eq5.3}
x_{21}^2 -\left(\frac{a}{b}x_{22} -1\right)x_{21} +\left(\frac{1}{b}x_{22}^2 -\frac{a}{b}x_{22} + 1\right)=0.
\end{equation}
We have seen in Lemma \ref{lem5.3} that the discriminant of this equation is $\frac{(x_{22}-a)^2}{\Delta}$. Since $\Delta$
is a non-square element, it follows that (\ref{eq5.3}) has solutions if and only if $x_{22}=a$. Clearly, (\ref{eq5.3}) has at least one solution because $A_3\in V(\p_1)$.  Hence, $x_{22}=a$. Substituting back to (\ref{eq5.3}) obtains
$$x_{21}^2 -\left(\frac{a}{b}\cdot a -1\right)x_{21} +\left(\frac{1}{b}\cdot a^2 -\frac{a}{b}\cdot a + 1\right)=0.$$
Together with the assumption that $\Delta=-b$, this equation can be simplified as 
$$x_{21}^2-2x_{21}+1=(x_{21}-1)^2=0.$$
Hence, $x_{21}=1$. Using $g_1=g_2=0$, we deduce that $x_{11}=0$ and $x_{12}=-b$. This implies that
$V(\p_1)=\{A_3\}.$ Therefore, $\D_{A_3}(\F_q)=\{0,A_3\}$ and $|\D_{A_3}(\F_q)|=2$.
\end{proof}

\begin{rem}{\rm
In Section \ref{sec6} below, we will give a discussion and conjecture on the geometry of $\D_{A_3}(\F_q)$ when $\Delta\neq -b$. 
\hbo}\end{rem}

\section{Examples} \label{sec6}
\setcounter{equation}{0}
\renewcommand{\theequation}
{6.\arabic{equation}}
\setcounter{theorem}{0}
\renewcommand{\thetheorem}
{6.\arabic{theorem}}

\noindent This last section contains some examples and remarks for small primes $q$.

\begin{exam}{\rm
Suppose that $\F_q=\F_2$ and $A_1=\begin{pmatrix}
    1  &  0  \\
    0  & 0 
\end{pmatrix}$. It follows from Proposition \ref{prop3.1} that $|\D_{A_1}(\F_2)|=8$. More precisely, 
$$\D_{A_1}(\F_2)=\left\{\begin{pmatrix}
    0  &  0  \\
    0  & 0 
\end{pmatrix},\begin{pmatrix}
    0  &  0  \\
    1  & 0 
\end{pmatrix},\begin{pmatrix}
    0  &  0  \\
    0  & 1 
\end{pmatrix},\begin{pmatrix}
    0  &  0  \\
    1  & 1 
\end{pmatrix} ,\begin{pmatrix}
    0  &  1  \\
    0  & 0 
\end{pmatrix}, \begin{pmatrix}
    0  &  1  \\
    0  & 1 
\end{pmatrix}, \begin{pmatrix}
    1  &  0  \\
    0  & 0 
\end{pmatrix}, \begin{pmatrix}
    1  &  0  \\
    0  & 1 
\end{pmatrix}\right\}.$$
Going back to Proposition \ref{prop3.1}, we can see that in this case, $D_1$ consists of the first four matrices; the fifth and the sixth matrices form the set
$D_2\setminus D_1$; and the last two matrices constitute $D_3$.
\hbo}\end{exam}

\begin{exam}{\rm
Suppose that $\F_q=\F_3$ and $A_1=\begin{pmatrix}
    1  &  0  \\
    0  & 2 
\end{pmatrix}$. As $\updelta=c_1^2-c_1c_2+c_2^2=1^2-1\cdot 2+2^2=3=0$ in $\F_3$.
By Proposition \ref{prop3.3} (1) that $\D_{A_1}(\F_3)$ consists of the following 8 matrices: 
$$\left\{\begin{pmatrix}
    0  &  0  \\
    0  & 0 
\end{pmatrix},\begin{pmatrix}
    1  &  0  \\
    0  & 0 
\end{pmatrix},\begin{pmatrix}
    0  &  0  \\
    0  & 2 
\end{pmatrix},\begin{pmatrix}
    1  &  0  \\
    0  & 2 
\end{pmatrix} ,\begin{pmatrix}
    1  &  0  \\
    1  & 2 
\end{pmatrix}, \begin{pmatrix}
    1  &  0  \\
    2  & 2
\end{pmatrix}, \begin{pmatrix}
    1  &  1  \\
    0  & 2 
\end{pmatrix}, \begin{pmatrix}
    1  &  2  \\
    0  & 2 
\end{pmatrix}\right\}.$$
Suppose that $\F_q=\F_7$ and $A_1=\begin{pmatrix}
    1  &  0  \\
    0  & 2 
\end{pmatrix}$. Note that $\updelta=3\neq 0$ in $\F_7$ and 
$$\frac{c_2^2}{c_2-c_1}=4\textrm{ and }\frac{c_1^2}{c_1-c_2}=-1=6.$$
Hence, Proposition \ref{prop3.3} (2) implies that the following 10 matrices:
 $$\left\{\begin{pmatrix}
    0  &  0  \\
    0  & 0 
\end{pmatrix},\begin{pmatrix}
    1  &  0  \\
    0  & 0 
\end{pmatrix},\begin{pmatrix}
    0  &  0  \\
    0  & 2 
\end{pmatrix},\begin{pmatrix}
    1  &  0  \\
    0  & 2 
\end{pmatrix} ,\begin{pmatrix}
    4  &  1  \\
    1  & 6 
\end{pmatrix}, \begin{pmatrix}
    4  &  2  \\
    4  & 6
\end{pmatrix}, \begin{pmatrix}
    4  &  3  \\
    5  & 6 
\end{pmatrix}, \begin{pmatrix}
    4  &  4  \\
    2  & 6 
\end{pmatrix}, \begin{pmatrix}
    4  &  5  \\
    3  & 6 
\end{pmatrix}, \begin{pmatrix}
    4  &  6  \\
    6  & 6 
\end{pmatrix}\right\}.$$
form the affine variety $\D_{A_1}(\F_7)$, where $ab=1$ in $E_3$ appeared in Proposition \ref{prop3.3} (2).
\hbo}\end{exam}

\begin{exam}{\rm
Suppose that $\F_q=\F_5$ and $A_2=\begin{pmatrix}
    2  &  1  \\
    0  & 2 
\end{pmatrix}$. By Proposition \ref{prop4.2}, we see that $|\D_{A_2}(\F_5)|=7$. In fact, 
$$\D_{A_2}(\F_5)=\left\{\begin{pmatrix}
    0  &  0  \\
    0  & 0 
\end{pmatrix},\begin{pmatrix}
    2  &  1  \\
    0  & 2
\end{pmatrix},\begin{pmatrix}
    0  &  1  \\
    1  & 4 
\end{pmatrix},\begin{pmatrix}
    1  &  4  \\
    1  & 3 
\end{pmatrix} ,\begin{pmatrix}
    2  &  0  \\
    1  & 2 
\end{pmatrix}, \begin{pmatrix}
    3  &  4  \\
    1  & 1 
\end{pmatrix}, \begin{pmatrix}
    4  &  1  \\
    1  & 0 
\end{pmatrix}\right\}.$$
Similarly, if $A_2=\begin{pmatrix}
    3  &  1  \\
    0  & 3 
\end{pmatrix}$, then the following 7 matrices
$$\left\{\begin{pmatrix}
    0  &  0  \\
    0  & 0 
\end{pmatrix},\begin{pmatrix}
    3  &  1  \\
    0  & 3
\end{pmatrix} ,\begin{pmatrix}
    2  &  4  \\
    1  & 4 
\end{pmatrix}, \begin{pmatrix}
    3  &  0  \\
    1  & 3 
\end{pmatrix}, \begin{pmatrix}
    4  &  4  \\
    1  & 2 
\end{pmatrix},\begin{pmatrix}
    0  &  1  \\
    1  & 1 
\end{pmatrix},\begin{pmatrix}
    1  &  1  \\
    1  & 0 
\end{pmatrix}\right\}$$
constitute the variety $\D_{A_2}(\F_5)$.
\hbo}\end{exam}

\begin{rem}\label{rem6.4}
{\rm
In understanding the geometry of $\D_{A_3}(\F_q)$, we have seen in Section \ref{sec5} that the relationship between $\Delta$ and $-b$ plays a key role. We denote by $\nabla_0(q)$ the set of all pairs $(a,b)$ such that $\Delta$ is a non-square in $\F_q$
and $\Delta=-b$ and denote by $\nabla_1(q)$ the set of all pairs $(a,b)$ such that $\Delta$ is a non-square in $\F_q$
and $\Delta\neq -b$. For small $q$, we have
\begin{equation*} 
\begin{aligned}
\nabla_0(3)&=\{(0,1)\} & \nabla_1(3)&=\{(1,2),(2,2)\} \\
\nabla_0(5)&=\{(1,2),(2,3),(3,3),(4,2)\} & \nabla_1(5)&=\{(0,2),(0,3),(1,1),(2,4),(3,4),(4,1)\} \\
\end{aligned}
\end{equation*}
and $\nabla_0(7)=\emptyset$ and $|\nabla_1(7)|= 21$. Moreover,
$$\nabla_0(11)=\{(1,4),(2,5),(3,3),(4,9),(5,1),(6,1),(7,9),(8,3),(9,5),(10,4)\}$$
and the cardinality of $\nabla_1(11)=45.$
\hbo}\end{rem}

We have attempted to describe the geometric structure of $\D_{A_3}(\F_q)$ and compute it cardinality for $(a,b)\in \nabla_1(q)$. Our experiments using  MAGMA \cite{BCP97} lead us to make the following conjecture.

\begin{conj}\label{conj6.5}
Assume that $\cha(\F_q)\neq 2$ and $\Delta=a^2-4b$ is a non-square in $\F_q$. 
If $\Delta\neq -b$, then $|\D_{A_3}(\F_q)|=q+3$.
\end{conj}

Using MAGMA, we also have been able to computationally  verify this conjecture  for values of
$(a,b)\in \nabla_1(q)$ for $q=3,5,7$ and $11$. 

\vspace{5mm}
\noindent \textbf{Appendix}. Here we provide a MAGMA code for Gr\"obner basis and normal form computation
appeared in Lemma \ref{lem5.4}.

\begin{lstlisting}%[caption={Gröbner basis and  computation in MAGMA}, label={code:GB}]
p := 5;
F<a> := GF(p);
b := a^2 / 3;

R<x11, x12, x21, x22> := PolynomialRing(F, 4);

f1 := a*x12*x21 - b*x11*x21 + b*x22 + x11*x12;
f2 := a*x21*x22 - a*x22 - b*x21^2 + x11*x22 - x12;
f3 := a*b*x22 + a*x12*x22 - b^2*x21 - b*x11*x22 + x12^2;
f4 := a^2*x22 - a*b*x21 + a*x12 - a*x22^2 - b*x11 + b*x21*x22 - x12*x22;

J := ideal<R | f1, f2, f3, f4>;
GJ := GroebnerBasis(J);
GJ;

g1 := x11 + x22 - a;
h1 := x11 + 2*x22 - a;

d := g1 * x11 * h1;
Nd := NormalForm(d, GJ);
Nd;
\end{lstlisting}

\vspace{5mm}
\noindent \textbf{Acknowledgements}.  The symbolic computation language MAGMA \cite{BCP97} (http://magma. maths.usyd.edu.au/) was very helpful. This research was supported by NNSF of China under grant No. 12561003.
The authors would like to thank the three anonymous referees and the editor for their careful reading, constructive comments, and suggestions.


\raggedright

\begin{bibdiv}
  \begin{biblist}
  
  
\bib{Bax72}{article}{
   author={Baxter, Rodney J.},
   title={Partition function of the eight-vertex lattice model},
   journal={Ann. Physics},
   volume={70},
   date={1972},
   pages={193--228},
}

\bib{BCP97}{article}{
   author={Bosma, Wieb},
   author={Cannon, John},
   author={Playoust, Catherine},
   title={The Magma algebra system. I. The user language},
   journal={J. Symbolic Comput.},
   volume={24},
   date={1997},
   number={3-4},
   pages={235--265},
}

\bib{BFP99}{article}{
   author={Brenti, Francesco},
   author={Fomin, Sergey},
   author={Postnikov, Alexander},
   title={Mixed Bruhat operators and Yang-Baxter equations for Weyl groups},
   journal={Internat. Math. Res. Notices},
   date={1999},
   number={8},
   pages={419--441},
}
  
\bib{Car73}{article}{
   author={Carlitz, Leonard},
   title={The number of solutions of certain matric equations over a finite
   field},
   journal={Math. Nachr.},
   volume={56},
   date={1973},
   pages={105--109},
}

\bib{CJdR10}{article}{
   author={Cedo, Ferran},
   author={Jespers, Eric},
   author={del Rio, Angel},
   title={Involutive Yang-Baxter groups},
   journal={Trans. Amer. Math. Soc.},
   volume={362},
   date={2010},
   number={5},
   pages={2541--2558},
}

\bib{CCY19}{article}{
   author={Chen, Dongmei},
   author={Chen, Zhibing},
   author={Yong, Xuerong},
   title={Explicit solutions of the Yang-Baxter-like matrix equation for a
   diagonalizable matrix with spectrum contained in $\{1,\upalpha,0\}$},
   journal={Appl. Math. Comput.},
   volume={348},
   date={2019},
   pages={523--530},
}
  
\bib{CZ23}{article}{
   author={Chen, Yin},
   author={Zhang, Xinxin},
   title={A class of quadratic matrix equations over finite fields},
   journal={Algebra Colloq.},
   volume={30},
   date={2023},
   number={1},
   pages={169--180},
}


  \bib{CR25}{article}{
   author={Chen, Yin},
   author={Ren, Shan},
   title={Modular matrix invariants under some transpose actions},
   journal={Submitted to Finite Fields Appl. (with a positive report)},
   date={2025},
   pages={\texttt{arXiv:2504.12179}},
}

  \bib{CZ26}{article}{
   author={Chen, Yin},
   author={Zhang, Runxuan},
   title={Invariant theory and coefficient algebras of Lie algebras},
   journal={J. Algebra},
   volume={689},
   date={2026},
   pages={87--111},
}


\bib{Coh25}{article}{
  author={Cohen, Boaz},
 title={The generalized Yang-Baxter matrix equation $A(XA)^m=X(AX)^m$ over $\GL(2,\F)$},
  journal={To appear in J. Math. Sci. (N.Y.)},
  pages={https://doi.org/10.1007/s10958-025-07753-w},
  year={2025},
}

\bib{CLO15}{book}{
   author={Cox, David A.},
   author={Little, John},
   author={O'Shea, Donal},
   title={Ideals, varieties, and algorithms},
   series={Undergraduate Texts in Mathematics},
   edition={4},
   publisher={Springer, Cham},
   date={2015},
}

\bib{DK15}{book}{
   author={Derksen, Harm},
   author={Kemper, Gregor},
   title={Computational invariant theory},
   series={Encyclopaedia of Mathematical Sciences},
   volume={130},
   edition={Second enlarged edition},
   publisher={Springer, Heidelberg},
   date={2015},
}

\bib{GZ24}{article}{
   author={Gan, Yudan},
   author={Zhou, Duanmei},
   title={Iterative methods based on low-rank matrix for solving the
   Yang-Baxter-like matrix equation},
   journal={Comput. Appl. Math.},
   volume={43},
   date={2024},
   number={4},
   pages={Paper No. 241, 19},
}

\bib{Gec03}{book}{
   author={Geck, Meinolf},
   title={An introduction to algebraic geometry and algebraic groups},
   series={Oxford Graduate Texts in Mathematics},
   volume={10},
   publisher={Oxford University Press, Oxford},
   date={2003},
}

\bib{Hiv00}{article}{
   author={Hivert, Florent},
   title={Hecke algebras, difference operators, and quasi-symmetric
   functions},
   journal={Adv. Math.},
   volume={155},
   date={2000},
   number={2},
   pages={181--238},
}


\bib{Hod57}{article}{
   author={Hodges, John H.},
   title={Some matrix equations over a finite field},
   journal={Ann. Mat. Pura Appl. (4)},
   volume={44},
   date={1957},
   pages={245--250},
}

\bib{Hod58}{article}{
   author={Hodges, John H.},
   title={The matrix equation $X^{2}-I=0$ over a finite field},
   journal={Amer. Math. Monthly},
   volume={65},
   date={1958},
   pages={518--520},
}

\bib{Hod64}{article}{
   author={Hodges, John H.},
   title={A bilinear matrix equation over a finite field},
   journal={Duke Math. J.},
   volume={31},
   date={1964},
   pages={661--666},
}

\bib{Hod66}{article}{
   author={Hodges, John H.},
   title={An Hermitian matrix equation over a finite field},
   journal={Duke Math. J.},
   volume={33},
   date={1966},
   pages={123--129},
}



\bib{HM25}{article}{
   author={Huang, Ting},
   author={Miao, Shuxin},
   title={An improved zeroing neural network model for solving the
   time-varying Yang-Baxter-like matrix equation},
   journal={J. Math. Anal. Appl.},
   volume={545},
   date={2025},
   number={1},
   pages={Paper No. 129095, 16},
}

\bib{Lu22}{article}{
   author={Lu, Linzhang},
   title={Manifold expressions of all solutions of the Yang--Baxter-like
   matrix equation for rank-one matrices},
   journal={Appl. Math. Lett.},
   volume={132},
   date={2022},
   pages={Paper No. 108175, 7},
}

\bib{PDK22}{article}{
   author={Peretz, Yossi},
   author={Dotan, Maya},
   author={Kamienny, Aytan},
  title={An algorithm for simultaneous nonsymmetric algebraic Riccati equations over finite fields},
  journal={J. Inf. Secur. Appl.},
  volume={67},
  pages={103178},
  year={2022},
}

\bib{WWW24}{article}{
   author={Wang, Yunjie},
   author={Wu, Cuilan},
   author={Wu, Gang},
   title={Solutions of the Yang-Baxter-like matrix equation with $3\times 3$
   diagonalizable coefficient matrix},
   journal={Linear Multilinear Algebra},
   volume={72},
   date={2024},
   number={14},
   pages={2347--2367},
}


\bib{WZ01a}{article}{
   author={Wei, Hongzeng},
   author={Zhang, Yibin},
   title={The number of solutions to the alternate matrix equation over a
   finite field and a $q$-identity},
   journal={J. Statist. Plann. Inference},
   volume={94},
   date={2001},
   number={2},
   pages={349--358},
}

\bib{WZ01b}{article}{
   author={Wei, Hongzeng},
   author={Zheng, Xingfen},
   title={The number of solutions to the bilinear matrix equation over a
   finite field},
   journal={J. Statist. Plann. Inference},
   volume={94},
   date={2001},
   number={2},
   pages={359--369},
}


\bib{Yan67}{article}{
   author={Yang, Chen-Ning},
   title={Some exact results for the many-body problem in one dimension with
   repulsive delta-function interaction},
   journal={Phys. Rev. Lett.},
   volume={19},
   date={1967},
   pages={1312--1315},
}

  \end{biblist}
\end{bibdiv}
\end{document}